\documentclass[11pt,reqno,twoside]{article}

\usepackage[hang]{footmisc}
\usepackage{lipsum}

\usepackage{fixltx2e} 

\usepackage{cmap} 

\usepackage[T1]{fontenc}
\usepackage[utf8]{inputenc}
\usepackage{graphicx}
\usepackage{placeins}
\usepackage{enumerate}
\usepackage{algcompatible}
\usepackage{subcaption}
\usepackage{verbatim}
\newcommand{\comments}[1]{}

\usepackage{soul}

\usepackage{setspace}

\usepackage{lmodern} 
\usepackage[scale=0.88]{tgheros} 

\usepackage{bm} 

\usepackage{bbold}

\usepackage{amsmath,amsbsy,amsgen,amscd,amsthm,amsfonts,amssymb} 

\numberwithin{equation}{section} 

\usepackage[centering,top=1.2in,bottom=1.2in,left=1in,right=1in]{geometry}

\usepackage{titling}
\usepackage{musicography}
\graphicspath{ {./images/} }

\usepackage[sf,bf,compact]{titlesec}

\usepackage{booktabs,longtable,tabu} 
\usepackage[font=small,margin=30pt,labelfont={sf,bf},labelsep={space}]{caption}

\usepackage[usenames,dvipsnames]{xcolor}

\definecolor{dark-gray}{gray}{0.3}
\definecolor{dkgray}{rgb}{.4,.4,.4}
\definecolor{dkblue}{rgb}{0,0,.5}
\definecolor{medblue}{rgb}{0,0,.75}
\definecolor{rust}{rgb}{0.5,0.1,0.1}

\usepackage{url}
\usepackage[colorlinks=true]{hyperref}
\hypersetup{linkcolor=dkblue}    
\hypersetup{citecolor=rust}      
\hypersetup{urlcolor=rust}     

\usepackage[final]{microtype} 

\newtheoremstyle{myThm} 
    {\topsep}                    
    {\topsep}                    
    {\itshape}                   
    {}                           
    {\sffamily\bfseries}                   
    {.}                          
    {.5em}                       
    {}  

\newtheoremstyle{myRem} 
    {\topsep}                    
    {\topsep}                    
    {}                   
    {}                           
    {\sffamily}                   
    {.}                          
    {.5em}                       
    {}  

\newtheoremstyle{myDef} 
    {\topsep}                    
    {\topsep}                    
    {}                   
    {}                           
    {\sffamily\bfseries}                   
    {.}                          
    {.5em}                       
    {}  

\theoremstyle{myThm}
\newtheorem{theorem}{Theorem}[section]

\newtheorem{proposition}[theorem]{Proposition}
\newtheorem{corollary}[theorem]{Corollary}

\theoremstyle{myRem}
\newtheorem{remarkx}[theorem]{Remark}
 \newenvironment{remark}
  {\pushQED{\qed}\remarkx}
  {\popQED\endremarkx}

\theoremstyle{myDef}

\usepackage{fancyhdr}
\let\originalleft\left
\let\originalright\right
\renewcommand{\left}{\mathopen{}\mathclose\bgroup\originalleft}
\renewcommand{\right}{\aftergroup\egroup\originalright}

\usepackage{mathtools}
\mathtoolsset{centercolon}  

\definecolor{mygreen}{rgb}{0.1,0.75,0.2}

\providecommand{\mathbbm}{\mathbb} 

\newcommand{\R}{\mathbbm{R}}

\newcommand{\F}{\mathcal{F}}

\renewcommand{\phi}{\varphi}

\usepackage[font = small, margin=30pt]{caption}

\usepackage[]{algorithm}
\usepackage{algpseudocode}
\usepackage{enumerate}

\usepackage{graphicx}

\usepackage{authblk}
\usepackage[square,numbers]{natbib}
\usepackage{chngcntr}
\usepackage{mathrsfs}
\counterwithin{table}{section}
\counterwithin{algorithm}{section}

\newcommand{\ttrace}{\text{Tr}}

\newcommand{\E}{\mathbb{E}}

\newcommand{\hatSigma}{\widehat{\Sigma}}

\newcommand{\mcE}{\mathcal{E}}

\newcommand{\mcN}{\mathcal{N}}

\newcommand{\mcS}{\mathcal{S}}

\usepackage{multirow}

\newcommand{\normn}[1]{\ensuremath{\lVert #1 \rVert}}

\newcommand{\inp}[1]{\ensuremath{\left\langle #1 \right\rangle}}

\usepackage[scr = esstix, cal = cm, frak=euler]{mathalfa}

\definecolor{mygreen}{rgb}{0.1,0.75,0.2}

\usepackage[scr = esstix, cal = cm, frak=euler]{mathalfa}

\title{On the Estimation of Gaussian Moment Tensors} 

\author{
Omar Al-Ghattas\textsuperscript{*} \ \ 
Jiaheng Chen\textsuperscript{\dag} \ \
Daniel Sanz-Alonso\textsuperscript{\ddag}
}

\vspace{.25in}

\makeatletter\@addtoreset{section}{part}\makeatother%

\newcommand{\upperRomannumeral}[1]{\uppercase\expandafter{\romannumeral#1}}

\date{}

\begin{document}

\maketitle 
\renewcommand{\thefootnote}{\fnsymbol{footnote}}

\footnotetext[1]{Eric and Wendy Schmidt Center, Broad Institute of MIT and Harvard}
\footnotetext[2]{Committee on Computational and Applied Mathematics, University of Chicago}
\footnotetext[3]{Department of Statistics, University of Chicago}


\vspace{-4em}
\abstract{This paper studies two estimators for Gaussian moment tensors: the standard sample moment estimator and a plug-in estimator based on Isserlis's theorem. We establish dimension-free, non-asymptotic error bounds that  demonstrate and quantify the advantage of Isserlis’s estimator for tensors of even order $p  >2.$ Our bounds hold in operator and entrywise maximum norms, and apply to symmetric and asymmetric tensors.} 



\vspace{1em}
\section{Introduction}\label{sec:Introduction}
This paper develops dimension-free, non-asymptotic theory for estimating Gaussian moment tensors of arbitrary order. We study two estimators: the sample moment estimator and a plug-in estimator based on Isserlis's formula. Our non-asymptotic bounds quantify the advantage of Isserlis's estimator for tensors of even order $p>2,$ both in operator and entrywise maximum norms. By studying the sample moment estimator, this work contributes to the emerging literature on concentration inequalities for random tensors \cite{tomioka2014spectral,vershynin2020concentration,zhou2021sparse,zhivotovskiy2024dimension, bandeira2024geometric, boedihardjo2024injective, al2025covariance,chen2025sharp}; and by studying Isserlis's estimator, it contributes to the literature on estimating functionals of high and infinite-dimensional parameters \cite{koltchinskii2018asymptotic,koltchinskii2021asymptotically,koltchinskii2021estimation, koltchinskii2025estimation}. 
More broadly, this work is motivated by the growing adoption of tensor methods in statistics and machine learning \cite{mccullagh2018tensor,lim2021tensors, Ballard_Kolda_2025, bi2021tensors, auddy2024tensors, cichocki2015tensor, sidiropoulos2017tensor},
which calls for new theory for high-dimensional tensor estimation.  

The concentration of the sample covariance (the sample moment tensor of order $p=2$) has been extensively studied. Most pertinent to our work, the papers \cite{vershynin2010introduction,koltchinskii2017concentration} identify a notion of \emph{effective dimension} that determines the sample complexity. In contrast to our results for even order $p>2,$ in the case $p=2$ the sample covariance matrix is known to be minimax optimal under the operator norm (see \cite[Theorem 2]{lounici2014high} and \cite[Theorem 4]{koltchinskii2017concentration}). Sharp operator norm concentration inequalities for simple random tensors of arbitrary order were recently established in \cite{al2025sharp,chen2025sharp} using the effective dimension from \cite{vershynin2010introduction,koltchinskii2017concentration}. In this paper, we prove a new sharp concentration inequality in entrywise maximum norm using a notion of effective dimension introduced in \cite{al2024optimal, al2025covariance} to analyze thresholding estimators for structured covariance operators. 

Isserlis's theorem expresses higher-order moments of a Gaussian distribution as functionals of its covariance. Consequently, our problem falls within the general framework of estimating functionals $f(\Sigma)$ of a covariance operator $\Sigma$ \cite{koltchinskii2018asymptotic,koltchinskii2021asymptotically}, where  plug-in estimators $f(\widehat{\Sigma})$ based on the sample covariance $\widehat{\Sigma}$ are often severely biased and suboptimal \cite{koltchinskii2021estimation, koltchinskii2025estimation}. Bias reduction techniques ---such as jackknife, bootstrap, and Taylor expansions--- are commonly used to improve efficiency.  
Our new dimension-free, non-asymptotic error bounds for Isserlis's estimator show that, in the context of estimating Gaussian moment tensors, plug-in estimation via Isserlis's formula performs well and strictly outperforms the sample moment estimator.

\section{Preliminaries}

Given positive integers $d_1,\dots, d_p$ and tensors $T,T' \in \R^{d_1 \times \cdots \times d_p}$, their Frobenius inner product is defined by
\begin{align*}
    \inp{T,T'} := 
    \sum_{i_1=1}^{d_1} \sum_{i_2=1}^{d_2}\cdots 
    \sum_{i_p=1}^{d_p}
    T_{i_1,\dots, i_p}T'_{i_1,\dots, i_p}.
\end{align*}
For vectors  $v_k \in \R^{d_k},$  $1\le k \le p,$ let $v_1 \otimes \cdots \otimes v_p \in \R^{d_1  \times \cdots \times d_p}$ denote their outer product, defined entrywise by
\begin{align*}
  (v_1 \otimes \cdots \otimes v_p)_{\ell_1,\dots, \ell_p} := v_{1, \ell_1} \cdots v_{p, \ell_p},  
\end{align*}
where $v_{i,\ell}$ denotes the $\ell$-th coordinate of $v_i.$ The operator norm of $T \in \R^{d_1 \times \cdots \times d_p}$ is
\begin{align*}
    \|T\| := 
    \sup_{v_k \in \mcS^{d_k}, 1\le k \le p 
    } \inp{T, v_1 \otimes \cdots \otimes v_p},
\end{align*}
where $\mcS^{d_k}:= \{ v \in \R^{d_k} : |v| = 1 \}$ denotes the unit Euclidean sphere in $\R^{d_k}.$  Similarly, the entrywise maximum norm is defined by
\begin{align*}
\|T\|_{\max} := \sup_{v_k \in \mcE^{d_k}, 1 \le k \le p}  |\inp{T, v_1 \otimes \cdots \otimes v_p}|,
\end{align*}
where $\mcE^{d_k} := \{e_i\}_{i=1}^{d_k}$ is the standard basis of $\R^{d_k}.$ 

We will consider two notions of effective dimension of a covariance matrix $\Sigma \in \R^{d\times d}.$ The first is 
\begin{align*}
r_2(\Sigma) := \frac{\ttrace(\Sigma)}{\|\Sigma\|},
\end{align*}
which will arise in our bounds in operator norm (see e.g. \cite{koltchinskii2017concentration, vershynin2010introduction}). The second is
\begin{align*}
r_{\max}(\Sigma) := \frac{\left( \E_{X \sim \mcN(0, \Sigma)} \|X\|_\infty \right)^2}{\|\Sigma\|_{\max}},
\end{align*}
which arises in our bounds in entrywise maximum norm (see e.g. \cite{ghattas2022non, al2024optimal, al2025covariance}).

For positive sequences $\{a_n\}, \{b_n\}$, we write $a_n \lesssim b_n$ to denote that, for some constant $c >0$, $a_n \le c b_n.$ If both $a_n \lesssim b_n$ and $b_n \lesssim a_n$ hold, we write $a_n \asymp b_n.$
We write $\lesssim_p$ and $\asymp_p$ to indicate that the implicit constant may depend on the parameter $p$.

\section{Main results}
This section contains our main results. In Subsection~\ref{ssec:symm}, we study the \emph{symmetric case}: moment tensors defined by expectation of the $p$-fold outer product of a single random vector. In Subsection~\ref{ssec:Asymm}, we consider the \emph{asymmetric case}: moment tensors defined by expectation of the outer product of $p$ random vectors of possibly different dimensions. 
While the asymmetric case subsumes the symmetric one, we present first the symmetric case due to its simplicity and its  central role in theory and applications.

\subsection{Symmetric case} \label{ssec:symm}
Let $X \sim \mcN(0, \Sigma)$ be a zero-mean Gaussian random vector in $\R^d$ with covariance matrix $\Sigma,$  and let $X_1, \ldots, X_N$ be i.i.d. copies of $X$. Let $p$ be an even integer. Our goal is to estimate the  $p$-th  order moment tensor
\begin{align}\label{eq:T}
    T :=  \E X^{\otimes p} := \E X \otimes \cdots \otimes X,
\end{align}
which is the expectation of the $p$-fold tensor product of $X$.

A natural estimator is the sample moment estimator $\widehat{T}_S$, defined by
\begin{align}\label{eq:T_S}
\widehat{T}_S:=\frac{1}{N} \sum_{i=1}^N X_i^{\otimes p}.
\end{align}
Entrywise, this corresponds to
\begin{align*}
\big(\widehat{T}_S\big)_{\ell_1, \ldots, \ell_p} 
:= \frac{1}{N} \sum_{i=1}^N X_{i, \ell_1} \cdots X_{i, \ell_p}.
\end{align*}
Alternatively,
Isserlis's theorem \cite{isserlis1918formula} (also known as Wick's probability theorem \cite{wick1950evaluation}) provides an exact expression for Gaussian moment tensors in terms of second-order moments. Specifically, for any multi-index $(\ell_1,\dots, \ell_p),$
\begin{align*}
(\E X^{\otimes p})_{\ell_1, \dots, \ell_p} 
=
\E[X_{\ell_1} \cdots X_{\ell_p}] = 
\sum_{\pi \in \Pi^2_p} \prod_{(j,k) \in \pi} \Sigma_{\ell_j, \ell_k},
\end{align*}
where $\Pi^2_p$ is the set of all pairwise partitions of $\{1, \dots, p\}$. This motivates Isserlis's estimator, which substitutes the sample covariance \begin{align*}
\hatSigma := \frac{1}{N} \sum_{i=1}^N X_i X_i^\top
\end{align*}
into the same expression.  The resulting estimator is given entrywise by 
\begin{align}\label{eq:T_I}
\big(\widehat{T}_I\big)_{\ell_1, \ldots, \ell_p} := 
\sum_{\pi \in \Pi^2_p} \prod_{(j,k) \in \pi}
\hatSigma_{\ell_j, \ell_k}.
\end{align}
Using the notion of induced likelihood \cite{zehna1966invariance},
the plug-in estimator $\widehat{T}_I$ can be interpreted as the maximum likelihood estimator of $T.$

Theorem \ref{thm:IsserlissBest} compares the performance of the two estimators, $\widehat{T}_{S}$ and $\widehat{T}_I$, under both the operator norm and the entrywise maximum norm.

\begin{theorem}\label{thm:IsserlissBest}
Let $X \sim \mcN(0, \Sigma)$ be a zero-mean Gaussian random vector in $\R^d$ with covariance matrix $\Sigma,$  and let $X_1, \ldots, X_N$ be i.i.d. copies of $X$. For any even integer $p$, let $T, \widehat{T}_{S},$ and $ \widehat{T}_{I}$ be as defined in equations \eqref{eq:T}, \eqref{eq:T_S}, and \eqref{eq:T_I}, respectively. Then the following bounds hold.

    \smallskip
    \noindent    (i) Operator norm bounds:
    \begin{align*}
     \E \normn{\widehat{T}_S- T} 
     &\asymp_p \|\Sigma\|^{p/2} \bigg(\sqrt{\frac{r_2(\Sigma)}{N}}+\frac{r_2(\Sigma)^{p/2}}{N}\bigg),\\
     \E \normn{ \widehat{T}_I- T} 
     &\lesssim_p \|\Sigma\|^{p/2} \bigg(\sqrt{\frac{r_2(\Sigma)}{N}}+\bigg(\frac{r_2(\Sigma)}{N}\bigg)^{p/2}\bigg).
    \end{align*}
    \noindent    (ii) Entrywise maximum norm bounds:
    \begin{align*}
     \E \normn{ \widehat{T}_S- T
     }_{\max}
     &\asymp_p \|\Sigma\|_{\max}^{p/2}\bigg(\sqrt{\frac{r_{\max}(\Sigma)}{N}}+\frac{r_{\max}(\Sigma)^{p/2}}{N}\bigg),\\
     \E \normn{\widehat{T}_I - T 
     }_{\max}
     &\lesssim_p 
     \|\Sigma\|_{\max}^{p/2} \bigg(\sqrt{\frac{r_{\max}(\Sigma)}{N}}+\bigg(\frac{r_{\max}(\Sigma)}{N}\bigg)^{p/2}\bigg).
    \end{align*}
\end{theorem}
\begin{proof}

    \textit{(i) Operator norm bounds.} The sample moment estimator bound follows by \cite[Theorem 2.1]{al2025sharp} (see also \cite{abdalla2025dimension}). For $\widehat{T}_I$, by Corollary~\ref{coro:perturbation}, we have
    \begin{align*}
    \frac{\|\widehat{T}_I-T\|}{\|T\|}
    \le \frac{p}{2}\cdot\frac{\|\widehat{\Sigma}-\Sigma\|}{\|\Sigma\|} \bigg(1+\frac{\|\widehat{\Sigma}-\Sigma\|}{\|\Sigma\|}\bigg)^{p/2-1}
    \lesssim_p  
    \frac{\|\widehat{\Sigma}-\Sigma\|}{\|\Sigma\|}+\bigg(\frac{\|\widehat{\Sigma}-\Sigma\|}{\|\Sigma\|}\bigg)^{p/2}.
    \end{align*}
    Further, by \cite[Theorem 4]{koltchinskii2017concentration}, $\E \| \hatSigma - \Sigma\|/\|\Sigma\| \asymp \sqrt{\frac{r_2(\Sigma)}{N}} + \frac{r_2(\Sigma)}{N},$ which can be plugged into the expression in the previous display to give the result.
    
    \smallskip
    
   \noindent \textit{(ii) Entrywise maximum norm bounds.} For the sample moment estimator, the upper bound is a direct corollary of Theorem~\ref{thm:IsserlissBest_asymmetry}. The lower bound follows by an analogous argument to the one used in the proof of \cite[Theorem 2.1]{al2025sharp} and is omitted for brevity. For Isserlis's estimator, the bound follows from Corollary~\ref{coro:perturbation}.
\end{proof}

\begin{remark}
    Theorem~\ref{thm:IsserlissBest} shows that consistent estimation of $T$ using the sample moment estimator $\widehat{T}_{S}$ requires a sample size  $N\gg (r_2(\Sigma))^{p/2}$ under the operator norm, or $N\gg (r_{\max}(\Sigma))^{p/2}$ under the entrywise maximum norm. In contrast, consistency of Isserlis's estimator $\widehat{T}_I$ only requires sample size $N\gg r_2(\Sigma)$ or $N\gg r_{\max}(\Sigma)$, leading to a significant reduction in sample complexity.
\end{remark}

\begin{remark}
    The upper bounds for the sample moment tensor in Theorem \ref{thm:IsserlissBest} hold for sub-Gaussian data. In contrast, our analysis of the Isserlis's estimator is limited to Gaussian data. An interesting question beyond the scope of this work is to leverage the generalization of Isserlis's theorem for isotropic distributions in \cite[Theorem 3]{munthe2025short} to define and analyze more general Isserlis's-type estimators. Numerical results in \cite{al2025covarianceB} suggest that the advantage of Isserlis's-type estimators over the sample moment tensor may carry over to isotropic, sub-Gaussian data.  
\end{remark}

The following theorem establishes a lower bound for Isserlis's estimator.

\begin{theorem}\label{coro:Isserlis_lower}
Let $X \sim \mcN(0, \Sigma)$ be a zero-mean Gaussian random vector in $\R^d$ with covariance matrix $\Sigma,$  and let $X_1, \ldots, X_N$ be i.i.d. copies of $X$. For any even integer $p$, let $T$ and $ \widehat{T}_{I}$ be as defined in equations \eqref{eq:T} and \eqref{eq:T_I}, respectively. Then the following bounds hold.
 \smallskip
    
    \noindent    (i) Operator norm bound:
\[
 \E \normn{ \widehat{T}_I- T} 
\gtrsim_p \|\Sigma\|^{p/2} \bigg(\frac{1}{\kappa(\Sigma)^{p/2-1}}\sqrt{\frac{r_2(\Sigma)}{N}}+\bigg(\frac{r_2(\Sigma)}{N}\bigg)^{p/2}\bigg),
\]
where $\kappa(\Sigma):=\lambda_{\max}(\Sigma)/\lambda_{\min}(\Sigma)$ is the condition number of $\Sigma$.

 \noindent    (ii) Entrywise maximum norm bound:
\[
 \E \normn{ \widehat{T}_I- T}_{\max} 
\gtrsim_p \|\Sigma\|_{\max}^{p/2} \bigg(\frac{1}{\kappa(D(\Sigma))^{p/2-1}}\sqrt{\frac{r_{\max}(\Sigma)}{N}}+\bigg(\frac{r_{\max}(\Sigma)}{N}\bigg)^{p/2}\bigg),
\]
where $D(\Sigma)$ is the diagonal matrix with the same diagonal entries as $\Sigma$, and $\kappa(D(\Sigma))$ denotes the condition number of $D(\Sigma)$.
\end{theorem}

\begin{proof}
   \textit{(i) Operator norm bound.} 
    By Proposition \ref{prop:pertubation_lower} below, we have
    \begin{align*}
    \frac{\|\widehat{T}_I-T\|}{\|T\|}\ge \max\left\{\frac{\|\widehat{\Sigma}-\Sigma\|}{\|\Sigma\|} \bigg(\frac{1}{\kappa(\Sigma)}\bigg)^{p/2-1}, \bigg(\frac{\|\widehat{\Sigma}-\Sigma\|}{\|\Sigma\|}\bigg)^{p/2}\right\}.
    \end{align*}
   Taking expectations on both sides and substituting the bound $\E \| \hatSigma - \Sigma\|/\|\Sigma\| \asymp \sqrt{\frac{r_2(\Sigma)}{N}} + \frac{r_2(\Sigma)}{N}$ from \cite[Theorem 4]{koltchinskii2017concentration} yields the desired result.

    \smallskip
    
   \noindent \textit{(ii) Entrywise maximum norm bound.}  By Proposition \ref{prop:pertubation_lower} below,
    \begin{align*}
    \frac{\|\widehat{T}_I-T\|_{\max}}{\|T\|_{\max}}\ge \max\left\{\frac{\|D(\widehat{\Sigma}-\Sigma)\|_{\max}}{\|D(\Sigma)\|_{\max}} \bigg(\frac{1}{\kappa(D(\Sigma))}\bigg)^{p/2-1}, \bigg(\frac{\|D(\widehat{\Sigma}-\Sigma)\|_{\max}}{\|D(\Sigma)\|_{\max}}\bigg)^{p/2}\right\}.
    \end{align*}
   A straightforward adaptation of the argument in \cite[Proposition 3.1]{chen2025sharp} and \cite[Proposition 3.1]{al2025sharp} ---replacing the operator norm with the entrywise maximum norm--- yields
\begin{align*}
\E \|D(\widehat{\Sigma}-\Sigma)\|_{\max}=\E \sup_{v\in \mcE^d} |\langle (\widehat{\Sigma}-\Sigma)v,v \rangle| \gtrsim \|\Sigma\|_{\max}\bigg(\sqrt{\frac{r_{\max}(\Sigma)}{N}}+\frac{r_{\max}(\Sigma)}{N}\bigg).
\end{align*}
Taking expectations in the previous inequality, substituting this bound, and noting that $\|D(\Sigma)\|_{\max}=\|\Sigma\|_{\max}$ completes the proof.
\end{proof}

The following proposition, proved in Section \ref{sec:proposition_lower_proof}, was used in the proof of Theorem \ref{coro:Isserlis_lower}.
\begin{proposition}\label{prop:pertubation_lower} 
Let $X\sim \mcN(0,\Sigma_X)$ and $Y\sim \mcN(0,\Sigma_Y)$, and let $T_X:=\E X^{\otimes p}, T_Y:=\E Y^{\otimes p}$. Then,
    \[
\frac{\|T_X-T_Y\|}{\|T_Y\|}\ge \max\left\{\frac{\|\Sigma_X-\Sigma_Y\|}{\|\Sigma_Y\|} \left(\frac{1}{\kappa(\Sigma_Y)}\right)^{p/2-1}, \left(\frac{\|\Sigma_X-\Sigma_Y\|}{\|\Sigma_Y\|}\right)^{p/2}\right\},
\]
where $\kappa(\Sigma_Y):=\lambda_{\max}(\Sigma_Y)/\lambda_{\min}(\Sigma_Y)$ is the condition number of $\Sigma_Y$.
Similarly, under the entrywise maximum norm,
\begin{align*}
\frac{\|T_X-T_Y\|_{\max}}{\|T_Y\|_{\max}}&\ge \max\left\{\frac{\|D(\Sigma_X-\Sigma_Y)\|_{\max}}{\|D(\Sigma_Y)\|_{\max}} \left(\frac{1}{\kappa(D(\Sigma_Y))}\right)^{p/2-1}, \left(\frac{\|D(\Sigma_X-\Sigma_Y)\|_{\max}}{\|D(\Sigma_Y)\|_{\max}}\right)^{p/2}\right\}, 
\end{align*}
where $D(\Sigma_X-\Sigma_Y)$ and $D(\Sigma_Y)$ denote the diagonal matrices with the same diagonal entries as $\Sigma_X-\Sigma_Y$ and $\Sigma_Y$, respectively. 
\end{proposition}

\subsection{Asymmetric case} \label{ssec:Asymm}

Let $p$ be an even integer, and let $X=(X^{(1)},\ldots,X^{(p)})\in \R^d$ be a zero-mean Gaussian random vector with covariance matrix $\Sigma$, where each block $X^{(k)}\in \R^{d_k}$ and $\sum_{k=1}^{p}d_k=d$. For each $k\in \{1,\ldots,p\}$, denote the marginal covariance by $\Sigma^{(k)}:=\E X^{(k)}\otimes X^{(k)}$, and the cross-covariance by $\Sigma^{(j,k)}:=\E X^{(j)}\otimes X^{(k)}$.  Let $X_1,\ldots,X_N$ be i.i.d. copies of $X$, where $X_i=(X^{(1)}_i,\ldots,X^{(p)}_i)$. Our goal is to estimate the moment tensor 
\begin{align}\label{eq:T_asym}
T:=\E X^{(1)}\otimes\cdots \otimes X^{(p)},
\end{align}
formed by taking the tensor product over the blocks of $X$.

A natural estimator is the sample moment estimator, defined as
\begin{align}\label{eq:T_S_asym}
\widehat{T}_S:=\frac{1}{N} \sum_{i=1}^N X_i^{(1)} \otimes \cdots \otimes X_i^{(p)}.
\end{align}
Entrywise, this corresponds to
\begin{align*}
\big(\widehat{T}_S\big)_{\ell_1, \ldots, \ell_p} 
:= \frac{1}{N} \sum_{i=1}^N X^{(1)}_{i, \ell_1} \cdots X^{(p)}_{i, \ell_p}.
\end{align*}
Alternatively, by Isserlis's theorem, for any multi-index $(\ell_1,\ldots,\ell_p)$,
\begin{align*}
\big(\E X^{(1)}\otimes\cdots \otimes X^{(p)}\big)_{\ell_1, \dots, \ell_p} =
\E[X^{(1)}_{\ell_1} \cdots X^{(p)}_{\ell_p}] = 
\sum_{\pi \in \Pi^2_p} \prod_{(j,k) \in \pi} \Sigma^{(j,k)}_{\ell_j, \ell_k},
\end{align*}
where $\Pi^2_p$ denotes the set of pairwise partitions of $\{1,\ldots,p\}$. This motivates Isserlis's estimator, which substitutes the sample covariances 
\[
\hatSigma^{(j,k)} := \frac{1}{N} \sum_{i=1}^N X^{(j)}_i (X^{(k)}_i)^\top
\]
into the same expression. The resulting estimator is given entrywise by
\begin{align}\label{eq:T_I_asym}
\big(\widehat{T}_I\big)_{\ell_1, \ldots, \ell_p} := 
\sum_{\pi \in \Pi^2_p} \prod_{(j,k) \in \pi}
\hatSigma^{(j,k)}_{\ell_j, \ell_k}.
\end{align}

Theorem \ref{thm:IsserlissBest_asymmetry} compares the performance of the two estimators, $\widehat{T}_{S}$ and $\widehat{T}_I$, under both the operator norm and the entrywise maximum norm.

\begin{theorem}\label{thm:IsserlissBest_asymmetry} 
Let $p$ be an even integer, and let $X=(X^{(1)},\ldots,X^{(p)})\in \R^d$ be a zero-mean Gaussian random vector with covariance matrix $\Sigma$. For each $k\in \{1,\ldots,p\}$, let $\Sigma^{(k)}:=\E X^{(k)}\otimes X^{(k)}$, and for $j,k\in \{1,\ldots,p\}$, let $\Sigma^{(j,k)}:=\E X^{(j)}\otimes X^{(k)}$. Let $X_1,\ldots,X_N$ be i.i.d. copies of $X$, where $X_i=(X^{(1)}_i,\ldots,X^{(p)}_i)$. Let $T, \widehat{T}_{S},$ and $ \widehat{T}_{I}$ be as defined in equations \eqref{eq:T_asym}, \eqref{eq:T_S_asym}, and \eqref{eq:T_I_asym}, respectively. Then the following bounds hold.

\smallskip

   \noindent (i) Operator norm bounds:  \begin{align*}
     \E \|\widehat{T}_S- T\|
     &\lesssim_p \bigg(\prod_{k=1}^p\|\Sigma^{(k)}\|^{1 / 2}\bigg)\Bigg(\bigg(\frac{\sum_{k=1}^p r_2(\Sigma^{(k)})}{N}\bigg)^{1/2}+\frac{\prod_{k=1}^{p}(r_2(\Sigma^{(k)})+\log N)^{1/2}}{N}\Bigg),\\
     \E \|\widehat{T}_I- T \| 
     &\lesssim_p\bigg(\prod_{k=1}^p\|\Sigma^{(k)}\|^{1 / 2}\bigg) \bigg(\frac{\max_{1\le k\le p}r_2(\Sigma^{(k)})}{N}\bigg)^{1/2},\ \ \text{if } \ N\ge \max_{1\le k\le p}r_2(\Sigma^{(k)}).
    \end{align*}
  \noindent    (ii) Entrywise maximum norm bounds:
  \begin{align*}
     \E \|\widehat{T}_S-T\|_{\max}
     &\lesssim_p \!\bigg(\prod_{k=1}^p\|\Sigma^{(k)}\|_{\max}^{1 / 2}\bigg)\Bigg(\bigg(\frac{\sum_{k=1}^p r_{\max}(\Sigma^{(k)})}{N}\bigg)^{1/2}\! \!+\frac{\prod_{k=1}^{p}(r_{\max}(\Sigma^{(k)})+\log N)^{1/2}}{N}\Bigg),\\
     \E \| \widehat{T}_I-T\|_{\max} 
     &\lesssim_p\!\bigg(\prod_{k=1}^p\|\Sigma^{(k)}\|_{\max}^{1 / 2}\bigg) \bigg(\frac{\max_{1\le k\le p}r_{\max}(\Sigma^{(k)})}{N}\bigg)^{1/2},\ \ \text{if } \ N\ge \max_{1\le k\le p}r_{\max}(\Sigma^{(k)}).
    \end{align*}
\end{theorem}

\begin{remark}
    The upper bounds on the deviation of sample moment estimator $\widehat{T}_{S}$ from its expectation are sharp under both the operator norm and entrywise maximum norm when the vectors $X^{(1)},\ldots, X^{(p)}$, together with their samples $(X^{(1)}_i)_{i=1}^{N},\ldots,(X^{(p)}_i)_{i=1}^N$, are mutually independent; see \cite[Theorem 2.1]{chen2025sharp} and Theorem \ref{thm:maximum_asymmetric}.  Consistent estimation of $T$ using the sample moment estimator $\widehat{T}_{S}$ requires a sample size satisfying  
    \[
    N\gg \prod_{k=1}^{p}(r_2(\Sigma^{(k)})+\log N)^{1/2} \quad \text{or} \quad   N\gg \prod_{k=1}^{p}(r_{\max}(\Sigma^{(k)})+\log N)^{1/2}
    \]
    under the operator norm and the entrywise maximum norm, respectively. In contrast, Isserlis's estimator $\widehat{T}_I$ only requires sample size
    \[
    N\gg \max_{1\le k\le p}r_{2}(\Sigma^{(k)}) \quad \text{or} \quad N\gg \max_{1\le k\le p}r_{\max}(\Sigma^{(k)})
    \]
   for consistency under the respective norms, leading to a significant reduction in sample complexity.
\end{remark}

\begin{proof}[Proof of Theorem \ref{thm:IsserlissBest_asymmetry}]

First, for the sample moment estimator $\widehat{T}_{S}$, the upper bound on its deviation from $T$ under the operator norm follows directly from \cite[Theorem 2.1]{chen2025sharp}, while the corresponding bound under the entrywise maximum norm is given by Theorem~\ref{thm:maximum_asymmetric} in this paper. It remains to analyze Isserlis's estimator $\widehat{T}_I$.

We apply the upper bound in Proposition~\ref{prop:perturbation}, which yields
  \[ 
  \|\widehat{T}_I-T\|\le \bigg(\prod_{k=1}^{p}\|\Sigma^{(k)}\|^{1/2}\bigg) (p-1)!! \cdot \frac{p}{2}\cdot \varepsilon_{*}(1+\varepsilon_{*})^{p/2-1},
\]
where
\[
\varepsilon_{*}:=\max_{j\ne k}\frac{\|\widehat{\Sigma}^{(j,k)}-\Sigma^{(j,k)}\|}{\big(\|\Sigma^{(j,j)}\| \|\Sigma^{(k,k)}\|\big)^{1/2}}.
\]

To control $\varepsilon_*$, we use the bound from \cite[Remark 2.1 and page 21]{chen2025sharp} for the sample cross-covariance. For $j,k\in \{1,\ldots,p\}$ and any $u\ge 0$, it holds with probability at least $1-\exp(-u^2)$ that
\[
\frac{\|\widehat{\Sigma}^{(j,k)}-\Sigma^{(j,k)}\|}{\big(\|\Sigma^{(j,j)}\|\|\Sigma^{(k,k)}\|\big)^{1/2}}\lesssim u\bigg(\frac{r_2(\Sigma^{(j)})+r_2(\Sigma^{(k)})}{N}\bigg)^{1/2}+u^2\frac{(r_2(\Sigma^{(j)})r_2(\Sigma^{(k)}))^{1/2}}{N}.
\]
Applying a union bound over all $j\ne k$ yields, with probability at least $1-p^2 \exp(-u^2)$,
\begin{align*}
\varepsilon_* &=\max_{j\ne k}\frac{\|\widehat{\Sigma}^{(j,k)}-\Sigma^{(j,k)}\|}{\big(\|\Sigma^{(j,j)}\|\|\Sigma^{(k,k)}\|\big)^{1/2}}\\
&\lesssim \max_{j\ne k} \left\{u\bigg(\frac{r_2(\Sigma^{(j)})+r_2(\Sigma^{(k)})}{N}\bigg)^{1/2}+u^2\frac{(r_2(\Sigma^{(j)})r_2(\Sigma^{(k)}))^{1/2}}{N}\right\}\\
&\lesssim u \bigg(\frac{\max_{1\le k\le p}r_2(\Sigma^{(k)})}{N}\bigg)^{1/2}+u^2 \frac{\max_{1\le k\le p}r_2(\Sigma^{(k)})}{N}.
\end{align*}
Substituting this bound into the inequality for $\|\widehat{T}_I-T\|$, we obtain, with probability at least $1-p^2 \exp(-u^2)$:
 \begin{align*}
  \|\widehat{T}_I-T\| &\le \bigg(\prod_{k=1}^{p}\|\Sigma^{(k)}\|^{1/2}\bigg) (p-1)!! \cdot \frac{p}{2}\cdot \varepsilon_{*}(1+\varepsilon_{*})^{p/2-1}\\
  &\lesssim_p \bigg(\prod_{k=1}^{p}\|\Sigma^{(k)}\|^{1/2}\bigg) (\varepsilon_*+\varepsilon_*^{p/2})\\
  &\lesssim_p \bigg(\prod_{k=1}^{p}\|\Sigma^{(k)}\|^{1/2}\bigg)\bigg(u \bigg(\frac{\max_{1\le k\le p}r_2(\Sigma^{(k)})}{N}\bigg)^{1/2}+u^p \bigg(\frac{\max_{1\le k\le p}r_2(\Sigma^{(k)})}{N}\bigg)^{p/2}\bigg).
 \end{align*}
Integrating the tail bound yields the following expectation bound: \begin{align*}
 \E \|\widehat{T}_I-T\| &\lesssim_p\bigg(\prod_{k=1}^{p}\|\Sigma^{(k)}\|^{1/2}\bigg)\bigg( \bigg(\frac{\max_{1\le k\le p}r_2(\Sigma^{(k)})}{N}\bigg)^{1/2}+ \bigg(\frac{\max_{1\le k\le p}r_2(\Sigma^{(k)})}{N}\bigg)^{p/2}\bigg)\\
 &\asymp_p\bigg(\prod_{k=1}^p\|\Sigma^{(k)}\|^{1 / 2}\bigg) \bigg(\frac{\max_{1\le k\le p}r_2(\Sigma^{(k)})}{N}\bigg)^{1/2},\quad \text{ if } \ N\ge \max_{1\le k\le p}r_2(\Sigma^{(k)}).
 \end{align*}
An analogous argument yields the corresponding bound under the entrywise maximum norm. This completes the proof.
\end{proof}

The following proposition is used in the proof of Theorem~\ref{thm:IsserlissBest_asymmetry}, and its proof is deferred to Section~\ref{sec:proposition_proof}.

\begin{proposition}\label{prop:perturbation} 

Let $p$ be an even integer, and let $X=(X^{(1)},\ldots,X^{(p)})$ and $Y=(Y^{(1)},\ldots,Y^{(p)})$ be zero-mean Gaussian random vectors in $\R^d$, where each $X^{(k)}, Y^{(k)}\in \R^{d_k}$ and $d=\sum_{k=1}^{p}d_k$. For $j, k\in \{1,\ldots,p\}$, denote the cross-covariance blocks by
\[
\Sigma^{(j,k)}_{X}=\E X^{(j)}\otimes X^{(k)}, \qquad \Sigma^{(j,k)}_{Y}=\E Y^{(j)}\otimes Y^{(k)}.
\] 
Let
\[
T_{X}:=\E X^{(1)}\otimes \cdots \otimes X^{(p)}, \qquad T_{Y}:=\E Y^{(1)}\otimes \cdots \otimes Y^{(p)},
\]
denote the moment tensors formed by taking tensor products over the blocks of $X$ and $Y$, respectively. 
Then,
    \[ \|T_{X}-T_{Y}\|\le \bigg(\prod_{k=1}^{p}\|\Sigma^{(k,k)}_{Y}\|^{1/2}\bigg) (p-1)!! \cdot \frac{p}{2}\cdot \varepsilon_{*}(1+\varepsilon_{*})^{p/2-1},
\]
where
\[
\varepsilon_{*}:=\max_{j\ne k}\frac{\|\Sigma^{(j,k)}_{X}-\Sigma^{(j,k)}_{Y}\|}{\big(\|\Sigma^{(j,j)}_{Y}\| \|\Sigma^{(k,k)}_{Y}\|\big)^{1/2}}.
\]
 For the entrywise maximum norm, it similarly holds that 
 \[ \|T_{X}-T_{Y}\|_{\max}\le \bigg(\prod_{k=1}^{p}\|\Sigma^{(k,k)}_{Y}\|_{\max}^{1/2}\bigg) (p-1)!! \cdot \frac{p}{2}\cdot \bar{\varepsilon}_{*}(1+\bar{\varepsilon}_{*})^{p/2-1},
\]
where
\[
\bar{\varepsilon}_{*}:=\max_{j\ne k}\frac{\|\Sigma^{(j,k)}_{X}-\Sigma^{(j,k)}_{Y}\|_{\max}}{\big(\|\Sigma^{(j,j)}_{Y}\|_{\max} \|\Sigma^{(k,k)}_{Y}\|_{\max}\big)^{1/2}}.
\]
\end{proposition}

\begin{corollary}\label{coro:perturbation} Let $X\sim \mcN(0,\Sigma_X)$ and $Y\sim \mcN(0,\Sigma_Y)$, and let $T_X:=\E X^{\otimes p}, T_Y:=\E Y^{\otimes p}$. Then,
    \[
 \frac{\|T_X-T_Y\|}{\|T_Y\|}\le \frac{p}{2}\cdot\frac{\|\Sigma_{X}-\Sigma_{Y}\|}{\|\Sigma_{Y}\|}\left(1+\frac{\|\Sigma_{X}-\Sigma_{Y}\|}{\|\Sigma_{Y}\|}\right)^{p/2-1}. 
\]
Similarly, under the entrywise maximum norm,
    \begin{align*}
        \frac{\|T_X-T_Y\|_{\max}}{\|T_Y\|_{\max}}\le \frac{p}{2}\cdot\frac{\|\Sigma_{X}-\Sigma_{Y}\|_{\max}}{\|\Sigma_{Y}\|_{\max}}\left(1+\frac{\|\Sigma_{X}-\Sigma_{Y}\|_{\max}}{\|\Sigma_{Y}\|_{\max}}\right)^{p/2-1}.
    \end{align*}
    \end{corollary}
    \begin{proof}[Proof of Corollary \ref{coro:perturbation}]
  Apply Proposition \ref{prop:perturbation} with $X^{(1)}=\cdots=X^{(p)}=:X$ and $Y^{(1)}=\cdots=Y^{(p)}=:Y$, so that $\Sigma_{X}^{(j,k)}=:\Sigma_{X}$ and  $\Sigma_{Y}^{(j,k)}=:\Sigma_{Y}$ for all $j,k\in \{1,\ldots,p\}$. The result then follows from computing the norm of the reference tensor:
        \[
       \| T_{Y}\|=\|\E Y^{\otimes p}\|=\sup_{v\in \mcS^{d}}\E \langle Y,v\rangle^p=(p-1)!! \cdot \|\Sigma_Y\|^{p/2},
        \]
     where the last equality uses the moment formula for centered Gaussian variables. The bound under the entrywise maximum norm follows analogously.       
    \end{proof}

\section{Proof of Proposition \ref{prop:pertubation_lower}}\label{sec:proposition_lower_proof}

\begin{proof}[Proof of Proposition \ref{prop:pertubation_lower}]
We begin by observing that
\begin{align*}
        \|T_X-T_Y\|
        &=\sup_{v\in \mcS^{d}} | \left\langle \E X^{\otimes p}- \E Y^{\otimes p},v^{\otimes p}\right\rangle|\\
        &=\sup_{v\in \mcS^{d}} \left| \E \langle X,v \rangle^{p}-\E \langle Y,v \rangle^{p}\right|\\
        &=(p-1)!!\sup_{v\in \mcS^{d}} \left| \langle\Sigma_X v,v\rangle^{p/2}-\langle \Sigma_Y v,v \rangle^{p/2}\right|\\
        &=(p-1)!!\sup_{v\in \mcS^{d}} \left| \langle (\Sigma_X- \Sigma_Y) v,v \rangle\right|\, \bigg|\sum_{\ell=0}^{p/2-1} \langle \Sigma_X v,v \rangle^{\ell}  \langle \Sigma_Y v,v \rangle^{p/2-1-\ell}\bigg|,        
    \end{align*}
  where the third equality uses the moment formula for centered Gaussian variables. Taking the supremum over the first term and the infimum over the second yields
  \begin{align*}
   \|T_X-T_Y\| &\ge (p-1)!! \sup_{v\in \mcS^{d}}\left| \langle (\Sigma_X- \Sigma_Y) v,v \rangle\right|\cdot \inf_{v\in \mcS^{d}}\bigg|\sum_{\ell=0}^{p/2-1} \langle \Sigma_X v,v \rangle^{\ell}  \langle \Sigma_Y v,v \rangle^{p/2-1-\ell}\bigg|\\
        &\ge (p-1)!!\|\Sigma_X-\Sigma_Y\| \cdot \lambda_{\min}(\Sigma_Y)^{p/2-1}.
  \end{align*}
Since $\|T_Y\| = (p-1)!!\,\|\Sigma_Y\|^{p/2}$, it follows that
    \begin{align}\label{eq:lower_aux1}
        \frac{\|T_X-T_Y\|}{\|T_Y\|} \ge \frac{\|\Sigma_X-\Sigma_Y\|}{\|\Sigma_Y\|} \cdot \left(\frac{1}{\kappa(\Sigma_Y)}\right)^{p/2-1},
    \end{align}
    where $\kappa(\Sigma_Y):=\lambda_{\max}(\Sigma_Y)/\lambda_{\min}(\Sigma_Y)$ is the condition number of $\Sigma_Y$. 

   Next, let $v$ be a unit eigenvector associated with the eigenvalue of $\Sigma_X - \Sigma_Y$ of largest magnitude, so that $\langle (\Sigma_X - \Sigma_Y)v, v \rangle = \pm \|\Sigma_X - \Sigma_Y\|$.  If the sign is positive, we have
\begin{align*}
    \|T_X-T_Y\|
    &\ge (p-1)!! \left| \langle\Sigma_X v,v\rangle^{p/2}-\langle \Sigma_Y v,v \rangle^{p/2}\right|\\
    &= (p-1)!! \left|\big(\langle (\Sigma_X-\Sigma_Y)v,v\rangle+\langle \Sigma_Y v,v \rangle\big)^{p/2}-\langle \Sigma_Y v,v \rangle^{p/2}\right| \\
    &= (p-1)!!\Big(\big(\|\Sigma_X-\Sigma_Y\|+\langle \Sigma_Y v, v\rangle\big)^{p/2}-\langle \Sigma_Y v,v \rangle^{p/2}\Big)\\
    &\ge (p-1)!! \|\Sigma_X-\Sigma_Y\|^{p/2}.
\end{align*}
If the sign is negative, a symmetric argument gives the same bound with $\Sigma_X$ and $\Sigma_Y$ interchanged. Therefore,
\begin{align}\label{eq:lower_aux2}
\frac{\|T_X-T_Y\|}{\|T_Y\|} \ge \frac{(p-1)!! \|\Sigma_X-\Sigma_Y\|^{p/2}}{(p-1)!!\|\Sigma_Y\|^{p/2}} = \left(\frac{\|\Sigma_X-\Sigma_Y\|}{\|\Sigma_Y\|}\right)^{p/2}.
\end{align}
Combining \eqref{eq:lower_aux1} and \eqref{eq:lower_aux2} yields the desired lower bound in operator norm.

For the entrywise maximum norm, we first obtain a lower bound of the same form as in the operator norm case, except that the supremum is now taken over the standard basis vectors. Specifically,
\begin{align*}
\|T_X-T_Y\|_{\max}&=\sup_{v_k \in \mcE^d, 1 \le k \le p}  |\inp{T_X-T_Y, v_1 \otimes \cdots \otimes v_p}|\\
&\ge \sup_{v \in \mcE^d}  |\inp{T_X-T_Y, v^{\otimes p}}|\\
&=(p-1)!!\sup_{v\in \mcE^d} \left| \langle\Sigma_X v,v\rangle^{p/2}-\langle \Sigma_Y v,v \rangle^{p/2}\right|,
\end{align*}
where $\mcE^{d} = \{e_i\}_{i=1}^{d}$ denotes the standard basis of $\R^d$. Then, the desired lower bound for the entrywise maximum norm follows from the same argument as for the operator norm case, with only minor modifications. We omit the details for brevity.
\end{proof}

\section{Proof of Proposition \ref{prop:perturbation}}\label{sec:proposition_proof}

\begin{proof}[Proof of Proposition \ref{prop:perturbation}]

For arbitrary vectors $v_{1}\in \R^{d_1},\ldots ,v_{p}\in \R^{d_p}$, we have
\begin{align*}
    \langle T_{X}-T_{Y},v_1\otimes \cdots \otimes v_p \rangle 
    &=\Big\langle \E X^{(1)}\otimes \cdots \otimes X^{(p)} - \E Y^{(1)}\otimes \cdots \otimes Y^{(p)},v_1\otimes \cdots \otimes v_p\Big\rangle\\
    &=\E \prod_{k=1}^{p}\langle X^{(k)},v_k \rangle-\E \prod_{k=1}^{p}\langle Y^{(k)},v_k \rangle\\
    &\overset{\text{(i)}}{=} \sum_{\pi\in \Pi_p^2} \prod_{(j,k)\in \pi}\E \langle X^{(j)},v_j \rangle \langle X^{(k)},v_k \rangle-\sum_{\pi\in \Pi_p^2} \prod_{(j,k)\in \pi}\E \langle Y^{(j)},v_j \rangle \langle Y^{(k)},v_k \rangle\\
    &=\sum_{\pi\in \Pi_p^2}\bigg(\prod_{(j,k)\in \pi}\langle \Sigma^{(j,k)}_{X} v_j,v_k \rangle-\prod_{(j,k)\in \pi} \langle \Sigma^{(j,k)}_{Y} v_j,v_k \rangle \bigg).
\end{align*}
Here, (i) follows from Isserlis's theorem, and $\Pi_p^2$ denotes the set of pairwise partitions of $\{1,\ldots,p\}$. Denoting by  $\pi=(\pi(1),\ldots,\pi(p))$ a fixed ordering of the indices in the pairing $\pi$, and applying the telescoping identity
\[
a_1 \cdots a_{p/2}-b_1 \cdots b_{p/2}=\sum_{\ell=1}^{p/2} a_1 \cdots a_{\ell-1}\left(a_\ell-b_\ell\right) b_{\ell+1} \cdots b_{p/2},
\]
we obtain
\begin{align*}
&\langle T_{X} - T_{Y}, v_1 \otimes \cdots \otimes v_p \rangle
= \sum_{\pi\in \Pi_p^2}\bigg(\prod_{(j,k)\in \pi}\langle \Sigma^{(j,k)}_{X} v_j,v_k \rangle-\prod_{(j,k)\in \pi} \langle \Sigma^{(j,k)}_{Y} v_j,v_k \rangle \bigg)\\
&=\sum_{\pi \in \Pi_p^2} \sum_{\ell=1}^{p/2} \Bigg[\bigg(\prod_{s=1}^{\ell-1} \langle \Sigma^{(\pi(2s-1),\pi(2s))}_{X} v_{\pi(2s-1)}, v_{\pi(2s)} \rangle
\bigg) \\
&\quad \times \langle (\Sigma^{(\pi(2\ell-1),\pi(2\ell))}_{X} - \Sigma^{(\pi(2\ell-1),\pi(2\ell))}_{Y}) v_{\pi(2\ell - 1)}, v_{\pi(2\ell)} \rangle  \bigg( \prod_{s=\ell+1}^{p/2} \langle \Sigma^{(\pi(2s-1),\pi(2s))}_{Y} v_{\pi(2s-1)}, v_{\pi(2s)} \rangle\bigg)\Bigg].
\end{align*}

To bound the operator norm, we take the supremum over $v_1\in \mcS^{d_1},\ldots,v_p\in\mcS^{d_p}$,
\begin{align*}
    &\|T_{X}-T_{Y}\| 
    =\sup_{ v_k\in \mcS^{d_k},1\le k\le p} |\langle T_{X}- T_{Y},v_1\otimes \cdots \otimes v_p \rangle| \\
    & = \sup_{ v_k\in \mcS^{d_k},1\le k\le p}\bigg|\sum_{\pi \in \Pi_p^2} \sum_{\ell=1}^{p/2} \bigg(\prod_{s=1}^{\ell-1} \langle \Sigma^{(\pi(2s-1),\pi(2s))}_{X} v_{\pi(2s-1)}, v_{\pi(2s)} \rangle
\bigg) \\
&\quad \times \langle (\Sigma^{(\pi(2\ell-1),\pi(2\ell))}_{X} - \Sigma^{(\pi(2\ell-1),\pi(2\ell))}_{Y}) v_{\pi(2\ell - 1)}, v_{\pi(2\ell)} \rangle  \bigg( \prod_{s=\ell+1}^{p/2} \langle \Sigma^{(\pi(2s-1),\pi(2s))}_{Y} v_{\pi(2s-1)}, v_{\pi(2s)} \rangle\bigg)\bigg|\\
 &\le\sum_{\pi\in \Pi_{p}^2}\sum_{\ell=1}^{p/2}  \Bigg[\sup_{ v_k\in \mcS^{d_k},1\le k\le p} \bigg| \bigg(\prod_{s=1}^{\ell-1} \langle \Sigma^{(\pi(2s-1),\pi(2s))}_{X} v_{\pi(2s-1)}, v_{\pi(2s)} \rangle
\bigg) \\
&\quad\times \langle (\Sigma^{(\pi(2\ell-1),\pi(2\ell))}_{X} - \Sigma^{(\pi(2\ell-1),\pi(2\ell))}_{Y}) v_{\pi(2\ell - 1)}, v_{\pi(2\ell)} \rangle 
 \bigg( \prod_{s=\ell+1}^{p/2} \langle \Sigma^{(\pi(2s-1),\pi(2s))}_{Y} v_{\pi(2s-1)}, v_{\pi(2s)}\rangle \bigg)\bigg|\Bigg]\\
   & =\sum_{\pi\in \Pi_{p}^2}\sum_{\ell=1}^{p/2} \Bigg[ \bigg(\prod_{s=1}^{\ell-1} \|\Sigma^{(\pi(2s-1),\pi(2s))}_{X}\|\bigg)
   \|\Sigma^{(\pi(2\ell-1),\pi(2\ell))}_{X}-\Sigma^{(\pi(2\ell-1),\pi(2\ell))}_{Y}\|
    \bigg(\prod_{s=\ell+1}^{p/2} \|\Sigma^{(\pi(2s-1),\pi(2s))}_{Y}\|\bigg) \Bigg].
\end{align*} 

For $j,k\in \{1,\ldots,p\}$, we introduce the normalized deviation $\varepsilon^{(j,k)}:=\frac{\|\Sigma^{(j,k)}_{X}-\Sigma^{(j,k)}_{Y}\|}{\big(\|\Sigma^{(j,j)}_{Y}\| \|\Sigma^{(k,k)}_{Y}\|\big)^{1/2}} $
and define $\varepsilon_{*}:=\max_{j\ne k}\varepsilon^{(j,k)}$. Using that $\|\Sigma^{(j,k)}_{Y}\|\le \|\Sigma^{(j,j)}_{Y}\|^{1/2}\|\Sigma^{(k,k)}_{Y}\|^{1/2}$, we obtain
\begin{align*}
    &\|T_{X}-T_{Y}\|
    \le \sum_{\pi\in \Pi_{p}^2}\sum_{\ell=1}^{p/2} \Bigg[\bigg(\prod_{s=1}^{\ell-1} \|\Sigma^{(\pi(2s-1),\pi(2s))}_{X}\|\bigg)\\
    &\quad \times \|\Sigma^{(\pi(2\ell-1),\pi(2\ell))}_{X}-\Sigma^{(\pi(2\ell-1),\pi(2\ell))}_{Y}\|\bigg(\prod_{s=\ell+1}^{p/2} \|\Sigma^{(\pi(2s-1),\pi(2s))}_{Y}\|\bigg)\Bigg] \\
    &\le \sum_{\pi\in \Pi_{p}^2}\sum_{\ell=1}^{p/2} \Bigg[ \bigg(\prod_{s=1}^{\ell-1} \|\Sigma^{(\pi(2s-1),\pi(2s-1))}_{Y}\|^{1/2} \|\Sigma^{(\pi(2s),\pi(2s))}_{Y}\|^{1/2}\big(1+\varepsilon^{(\pi(2s-1),\pi(2s))}\big)\bigg)\\
    &\quad\times \bigg(\|\Sigma^{(\pi(2\ell-1),\pi(2\ell-1))}_Y\|^{1/2}\|\Sigma^{(\pi(2\ell),\pi(2\ell))}_Y\|^{1/2}\varepsilon^{(\pi(2\ell-1),\pi(2\ell))}\bigg)\\
    &\quad\times \bigg(\prod_{s=\ell+1}^{p/2} \|\Sigma^{(\pi(2s-1),\pi(2s-1))}_{Y}\|^{1/2}\|\Sigma^{(\pi(2s),\pi(2s))}_{Y}\|^{1/2}\bigg) \Bigg]\\
    &=\bigg(\prod_{k=1}^{p}\|\Sigma^{(k,k)}_{Y}\|^{1/2}\bigg)\sum_{\pi\in \Pi_{p}^2}\sum_{\ell=1}^{p/2}\varepsilon^{(\pi(2\ell-1),\pi(2\ell))}\prod_{s=1}^{\ell-1}\big(1+\varepsilon^{(\pi(2s-1),\pi(2s))}\big).
\end{align*}
Using the identity \[
\sum_{\ell=1}^{p/2}a_{\ell}\prod_{s=1}^{\ell-1}(1+a_s)=\prod_{\ell=1}^{p/2}(1+a_{\ell})-1,\]
we conclude that
\begin{align*}
    \|T_{X}-T_{Y}\|&\le\bigg(\prod_{k=1}^{p}\|\Sigma^{(k,k)}_{Y}\|^{1/2}\bigg)\sum_{\pi\in \Pi_{p}^2}\bigg(\prod_{\ell=1}^{p/2}\big(1+\varepsilon^{(\pi(2\ell-1),\pi(2\ell))}\big)-1\bigg)\\
    &\le  \bigg(\prod_{k=1}^{p}\|\Sigma^{(k,k)}_{Y}\|^{1/2}\bigg) (p-1)!! \left((1+\varepsilon_{*})^{p/2}-1\right)\\
    & \le  \bigg(\prod_{k=1}^{p}\|\Sigma^{(k,k)}_{Y}\|^{1/2}\bigg) (p-1)!! \cdot \frac{p}{2}\cdot \varepsilon_{*}(1+\varepsilon_{*})^{p/2-1},
\end{align*}
where we used the inequality $(1+\varepsilon_*)^{p/2}-1\le \frac{p}{2}\cdot\varepsilon_{*}(1+\varepsilon_{*})^{p/2-1}$.

The bound under the entrywise maximum norm follows from an analogous argument, replacing the supremum over the unit spheres $\mcS^{d_1},\ldots,\mcS^{d_p}$ with the supremum over the standard bases $\mcE^{d_k}=\{e_i\}_{i=1}^{d_k}$ for $1\le k\le p$. 
This completes the proof.
\end{proof}

\section{Entrywise maximum norm bound for sample moment tensor}

\begin{theorem}\label{thm:maximum_asymmetric}
  For any integer $p\ge 2$ and $1 \leq k \leq p$, let $X^{(k)}, X_1^{(k)}, \ldots, X_N^{(k)}$ be i.i.d. zero-mean Gaussian random vectors in $\R^{d_k}$ with covariance matrix $\Sigma^{(k)}$. Then,
\[
\E \bigg\|\frac{1}{N} \sum_{i=1}^N X_i^{(1)} \otimes \cdots \otimes X_i^{(p)}-\E X^{(1)} \otimes \cdots \otimes X^{(p)}\bigg\|_{\max} \lesssim_p \bigg(\prod_{k=1}^p\|\Sigma^{(k)}\|_{\max}^{1 / 2}\bigg)\mathscr{E}_N\big((\Sigma^{(k)})_{k=1}^p\big),
\]
where
\begin{align*}
\mathscr{E}_N\big((\Sigma^{(k)})_{k=1}^p\big):= \bigg(\frac{\sum_{k=1}^p r_{\max}(\Sigma^{(k)})}{N}\bigg)^{1/2}+\frac{1}{N}\prod_{k=1}^{p}\Big(r_{\max}(\Sigma^{(k)})+\log N\Big)^{1/2}.
\end{align*}
Moreover, the upper bound is sharp in the following two cases:
\begin{enumerate}
    \item \textit{Independent Components.} If $X^{(1)},\ldots, X^{(p)}, (X^{(1)}_i)_{i=1}^{N},\ldots,(X^{(p)}_i)_{i=1}^N$ are mutually independent, then
\[
\E \bigg\|\frac{1}{N} \sum_{i=1}^N X_i^{(1)} \otimes \cdots \otimes X_i^{(p)}-\E X^{(1)} \otimes \cdots \otimes X^{(p)}\bigg\|_{\max} \asymp_p\bigg(\prod_{k=1}^p\|\Sigma^{(k)}\|_{\max}^{1 / 2}\bigg)\mathscr{E}_N\big((\Sigma^{(k)})_{k=1}^p\big).
\]
\item \textit{Identical Components.} If $X^{(1)}=\cdots = X^{(p)}=X$ and $X_i^{(1)}=\cdots = X_i^{(p)}=X_i \,$ for all $1\le i\le N$, $\Sigma^{(1)}=\cdots=\Sigma^{(p)}=\Sigma,$ then
\begin{align*}
\E \bigg\|\frac{1}{N}\sum_{i=1}^{N} X_i^{\otimes p}-\E X^{\otimes p}\bigg\|_{\max} &\asymp_p \|\Sigma\|_{\max}^{p/2} \bigg(\sqrt{\frac{r_{\max}(\Sigma)}{N}}+\frac{(r_{\max}(\Sigma)+\log N)^{p/2}}{N}\bigg).
\end{align*}
Furthermore, since $(\log N)^{p/2}/N\lesssim_p 1/\sqrt{N}$, this bound simplifies to
\[
\E \bigg\|\frac{1}{N}\sum_{i=1}^{N} X_i^{\otimes p}-\E X^{\otimes p}\bigg\|_{\max}  \asymp_p \|\Sigma\|_{\max}^{p/2} \bigg(\sqrt{\frac{r_{\max}(\Sigma)}{N}}+\frac{r_{\max}(\Sigma)^{p/2}}{N}\bigg).
\]
\end{enumerate}
\end{theorem}

\begin{remark}
The upper bound in Theorem~\ref{thm:maximum_asymmetric} holds without requiring independence between the sequences $(X^{(k)}_i)_{i=1}^{N}$ and $(X^{(k')}_i)_{i=1}^{N}$ for $k\ne k'$; that is, no assumptions are made on the correlation structure across components.  Moreover, the upper bound in Theorem~\ref{thm:maximum_asymmetric} extends directly to sub-Gaussian settings. 
Dimension-dependent counterparts can also be derived by applying standard $\varepsilon$-net arguments in conjunction with the $\alpha$-sub-exponential concentration inequality of \cite{gotze2021concentration}.
\end{remark}
\begin{proof}[Proof of Theorem~\ref{thm:maximum_asymmetric}]

\textit{Upper bound} 
By the definition of the entrywise maximum norm, we have
\begin{align*}
    &\E \bigg\|\frac{1}{N} \sum_{i=1}^N X_i^{(1)} \otimes \cdots \otimes X_i^{(p)}-\E X^{(1)} \otimes \cdots \otimes X^{(p)}\bigg\|_{\max}\\
    &=\E \max_{v_k\in \mcE^{d_k},1\le k\le p} \bigg|\frac{1}{N}\sum_{i=1}^{N}\prod_{k=1}^{p}\langle X^{(k)}_i,v_k\rangle -\E \prod_{k=1}^{p} \langle X^{(k)}, v_k \rangle\bigg|,
\end{align*}
where $\mcE^{d_k} = \{e_i\}_{i=1}^{d_k}$ denotes the standard basis of $\R^{d_k}$. Let $\bar{\mcE}^{d_k}:=\mcE^{d_k}\cup -\mcE^{d_k}$, and define $\F^{(k)}:=\{\langle\cdot, v\rangle: v \in \bar{\mcE}^{d_k} \}$ for $1\le k\le p$. Then the maximum can be upper bounded by
\[
\E \sup_{f^{(k)} \in \mathcal{F}^{(k)},1\le k\le p}\bigg|\frac{1}{N} \sum_{i=1}^N \prod_{k=1}^p f^{(k)} (X^{(k)}_i)-\mathbb{E} \prod_{k=1}^p f^{(k)} (X^{(k)})\bigg|.
\]
Applying \cite[Theorem 2.2]{chen2025sharp}, we obtain
\begin{align}\label{eq:aux1}
&\E \!\sup_{f^{(k)} \in \mathcal{F}^{(k)},1\le k\le p}\bigg|\frac{1}{N} \sum_{i=1}^N \prod_{k=1}^p f^{(k)} (X^{(k)}_i)-\mathbb{E} \prod_{k=1}^p f^{(k)} (X^{(k)})\bigg| \nonumber\\
&\lesssim_p  \bigg(\prod_{k=1}^{p} d_{\psi_2}(\F^{(k)})\bigg) \bigg(\frac{\sum_{k=1}^{p}\bar{\gamma}(\F^{(k)},\psi_2)}{\sqrt{N}}+\frac{\prod_{k=1}^{p}\big(\bar{\gamma}(\F^{(k)},\psi_2)+(\log N)^{1/2}\big)}{N}\bigg) ,
\end{align}
where $\bar{\gamma}(\F^{(k)},\psi_2):=\gamma(\F^{(k)},\psi_2)/d_{\psi_2}(\F^{(k)})$. Here, $\gamma(\F,\psi_2)$ denotes Talagrand's generic chaining complexity of the function class $\F$ \cite[Definition 2.7.3]{talagrand2022upper}, $d_{\psi_2}(\F) := \sup_{f \in \F}\|f\|_{\psi_2}$, and $\psi_2$ refers to the Orlicz norm with Orlicz function $\psi(x) = e^{x^2}-1,$ see e.g. \cite[Definition 2.5.6]{vershynin2018high}.

Since $X^{(k)}$ is Gaussian with covariance $\Sigma^{(k)}$, the $\psi_2$-norm of linear functionals is equivalent to the $L_2$-norm. Hence, \begin{align*}
d_{\psi_2}(\F^{(k)})=\sup_{f^{(k)}\in\mathcal{F}^{(k)}}\|f^{(k)}\|_{\psi_2}\asymp \sup_{f^{(k)}\in\mathcal{F}^{(k)}}\|f^{(k)}\|_{L_2}=\sup_{v\in\bar{\mcE}^{d_k}}\big(\E\langle X^{(k)},v\rangle^2\big)^{1/2}=\|\Sigma^{(k)}\|_{\max}^{1/2}.
\end{align*}
For the generic chaining term, note that the canonical metric on $\bar{\mcE}^{d_k}$ is given by
\[
d_{X^{(k)}}(u, v)
:= \big(\mathbb{E}(\langle X^{(k)}, u\rangle-\langle X^{(k)}, v\rangle)^2\big)^{1/2}=\langle u-v, \Sigma^{(k)} (u-v) \rangle^{1/2}
=\|\langle\cdot, u\rangle-\langle\cdot, v\rangle\|_{L_2(\mu^{(k)})},
\]
where $\mu^{(k)}$ is the law of $X^{(k)}$. By Talagrand's majorizing measure theorem \cite[Theorem 2.10.1]{talagrand2022upper},
\[
\gamma (\F^{(k)}, \psi_2)\asymp \gamma (\F^{(k)}, L_2)=\gamma(\bar{\mcE}^{d_k}, d_{X^{(k)}}) \asymp \mathbb{E} \sup_{ u \in \bar{\mcE}^{d_k}}\langle X^{(k)}, u\rangle = \E \|X^{(k)}\|_{\infty}.
\]
We conclude that
 \[
 d_{\psi_2}(\F^{(k)}) \asymp \|\Sigma^{(k)}\|_{\max}^{1/2},\quad \gamma(\F^{(k)},\psi_2)\asymp \E \|X^{(k)}\|_{\infty},
 \]
 and therefore,
 \[
\E \bigg\|\frac{1}{N} \sum_{i=1}^N X_i^{(1)} \otimes \cdots \otimes X_i^{(p)}-\E X^{(1)} \otimes \cdots \otimes X^{(p)}\bigg\|_{\max} \lesssim_p \bigg(\prod_{k=1}^p\|\Sigma^{(k)}\|_{\max}^{1 / 2}\bigg)\mathscr{E}_N\big((\Sigma^{(k)})_{k=1}^p\big),
\]
where
\begin{align*}
\mathscr{E}_N\big((\Sigma^{(k)})_{k=1}^p\big):= \bigg(\frac{\sum_{k=1}^p r_{\max}(\Sigma^{(k)})}{N}\bigg)^{1/2}+\frac{1}{N}\prod_{k=1}^{p}\Big(r_{\max}(\Sigma^{(k)})+\log N\Big)^{1/2}
\end{align*}
and $r_{\max}(\Sigma^{(k)}) := \left( \E_{X^{(k)} \sim \mcN(0, \Sigma^{(k)})} \|X^{(k)}\|_\infty \right)^2/\|\Sigma^{(k)}\|_{\max}$.

\vspace{1em}
\textit{Lower bound}  The lower bounds follow from straightforward modifications of the argument in \cite[Proposition 3.1]{chen2025sharp} and \cite[Proposition 3.1]{al2025sharp}, respectively, by replacing the operator norm with the entrywise maximum norm. We omit the details for brevity.
\end{proof}

\section*{Acknowledgments}
The work of DSA was partly funded by NSF CAREER DMS-2237628. The work of OAG was supported in part by funding from the Eric and Wendy Schmidt Center at the Broad Institute of MIT and Harvard.

\bibliographystyle{plain}
\bibliography{references}

\end{document}